\documentclass{tac}
\usepackage[all,ps,arc]{xy}  
\usepackage{amsmath,amssymb,latexsym}
\newcommand{\B}{{\mathbb B}}
\newcommand{\C}{{\mathbb C}}

\newcommand{\K}{{\overline {K}}}

\newcommand{\point}{\ensuremath{\xymatrix{A\ar@<+.6ex>[r]^(.5){\alpha}
&B\ar@<+.6ex>[l]^(.5){\beta}}}}
\newcommand{\rg}{\ensuremath{\xymatrix{A\ar@<+1ex>[r]^{\alpha}\ar@<-1ex>[r]_{\gamma}&B\ar[l]|{\beta}}}}

\newtheorem{Theorem}{Theorem}[section]
\newtheorem{Lemma}[Theorem]{Lemma}
\newtheorem{Proposition}[Theorem]{Proposition}
\newtheorem{Definition}[Theorem]{Definition}

\newtheorem{Corollary}[Theorem]{Corollary}

\theoremstyle{remark}
\theoremstyle{plain}

\newtheorem{Example}[Theorem]{Example}

\copyrightyear{2010}


\dedication{Draft version}

\address{Dipartimento di Matematica, Universit\`a degli Studi di
Milano, Via C. Saldini, 50, 20133  Milano, Italia}

\eaddress{sandra.mantovani@unimi.it \CR giuseppe.metere@unimi.it}

\begin{document}

\newenvironment{changemargin}[2]{\begin{list}{}{
\setlength{\topsep}{0pt}
\setlength{\leftmargin}{0pt}
\setlength{\rightmargin}{0pt}
\setlength{\listparindent}{\parindent}
\setlength{\itemindent}{\parindent}
\setlength{\parsep}{0pt plus 1pt}
\addtolength{\leftmargin}{#1}\addtolength{\rightmargin}{#2}
}\item}{\end{list}}

\title{Normalities and Commutators}
\author{Sandra Mantovani, Giuseppe Metere}


\maketitle

\begin{abstract}
We first compare several algebraic notions of normality, from a categorical viewpoint. Then we introduce an
intrinsic description of Higgins' commutator for ideal-determined categories, and we define a new notion of normality in terms of
this commutator. Our main result is to extend to any semi-abelian category
the following well-known characterization of normal subgroups: a subobject $K$ is normal in $A$ if, and only if, $[A,K]\leq K$.
\end{abstract}

\section{Introduction}
The notion of \emph{ideal} is central in many algebraic
disciplines. As recalled in \cite{JMU07}, in universal algebra
ideals were introduced by Higgins \cite{Higg_commutator} for
$\Omega$-groups, and by Magari \cite{Magari} in a more general
setting (where not all ideals have to be  kernels). Subsequently,
Agliano and Ursini \cite{Agliano_Ursini} defined \emph{clots}, a
 notion lying in-between kernels and ideals.

It was only in 2007 that  Janelidze, M\'arki and Ursini
\cite{JMU07} reviewed from a categorical perspective the
relationship among normal subobjects, clots and ideals. They
described these three concepts in an intrinsic setting, by using
the notion of internal action
\cite{BJ_PROT_DESC_SEMIDIRECT_P,BJK_INT_OBJ_ACTS}. They showed
that the relationships
$$
N(A)\subseteq C(A) \subseteq I(A)
$$
among the sets of normal subobjects, clots and ideals in an object
$A$, become equalities in the semi-abelian case (categorical
counterpart of the BIT-speciale varieties of \cite{BIT-Speciale}).

Actually, in the varietal case, BIT-speciale axioms are more than
enough in order to guarantee $N(A)=I(A)$: this equality holds also
in the weaker context of BIT varieties of \cite{U_Bit}, also
called \emph{ideal-determined} in \cite{Gumm_Ursini}. In
\cite{JMTU} the authors showed that the categorical counterpart of
BIT varieties can be obtained by removing the so-called Hofmann's
Axiom from the old-style definition of a semi-abelian category
\cite{JMT_SEMIAB}. They called these categories ideal-determined.

In this paper, we first revisit various notions of normalities
from kernels to ideals,  pointing out the relationships exiting
among them in appropriate  categorical contexts. Then we move in
the ideal-determined case, where all those notions collapse, since
 the images of kernels along regular epimorphisms are again
 kernels. This fact allows us to formulate  the categorical notion
 of Higgins' commutator. The last is obtained
 by taking (through the realization map, see \ref{realization}) the regular image  of the
\emph{formal commutator},   internal interpretation
of the \emph{commutator words} of \cite{Higg_commutator}.  By following this
approach, we revisit also Huq's commutator \cite{Huq68}, showing
that in an ideal-determined unital category  \cite{bibbia}, Huq's commutator
$[H,K]_Q$ in $A$ is nothing but the normalization in $A$ of Higgins' commutator $[H,K]_H$.
The two commutators are different in general, even in the category of groups, if $H$ and $K$
are not  normal in $A$, as Example \ref{Example:Higg_not_Huq} shows.

Nevertheless they always coincide when $H \vee K=A$, in particular if one of the two subobjects is the whole $A$.
In this case we can freely refer to \emph{the} commutator $[A,K]$ (notice that when
$K$ is normal, the last coincides also with the normalization of Smith's commutator, as shown in \cite{GVdL}).
The case $H=A$ is special even because in this circumstance the commutator behaves well
w.r.t. normalization, in the sense that $[A,K]=[A,\overline{K}]$, as shown in Proposition \ref{kbar}.

In the category of groups, the commutator $[A,K]$ can be used to test wether the subgroup $K$
of $A$ is normal in $A$. Actually $K$ is normal in $A$ if, and only if, $[A,K]$ is a
subgroup of $K$. This characterization of normal subgroups is interesting even because it
establish a link between normality and commutators.
A natural question is to ask if the internal formulation of this connection is still valid
in our settings.

If a category $\C$ is ideal-determined and unital, Proposition \ref{neces} shows that any normal subobject $K$ of $A$
contains the commutator $[A,K]$.
In order to get the converse, we need to use one more ingredient, namely  Hoffmann's Axiom, which makes
$\C$ into a semi-abelian category. In this context, by means of Proposition \ref{kbar}, we can use
some results by Bourn and Gran on central extensions \cite{BG3,BG2,BG1}  proving this way that
the full characterization of normality via commutators holds an any semi-abelian category.
This result extends what happens for groups to rings, Lie algebras, Leibniz algebras, more
generally any variety of $\omega$-groups, as well as to Heyting algebras to the dual category
of the category of pointed sets.

\section{Normalities}\label{sec:normalities}
We present  some  different  notions of normality,
and try to explain the relationship among them. Although some of those can be
defined with few requirements on the base category $\C$, we will assume throughout the
paper that $\C$ is (at least) a pointed category with finite limits.
Finally, these notions are presented  from the strongest to the weakest.

\subsection*{Kernels}  This is classical. A map $k:K\to A$ is a \emph{kernel} when there exists a
map $f:A\to B$ such that the following is a pullback diagram
$$
\xymatrix{
K\ar[r]\ar[d]_{k}\ar@{}[dr]|(.3){\lrcorner}
&0\ar[d]
\\
A\ar[r]_{f}&B.}
$$
In other words, a kernel is the fiber
over the zero-element of the codomain of a certain morphism.
Notice that $k$ is a monomorphism, so that
kernels are indeed sub-objects.

\subsection*{Normal subobjects}
The categorical  notion of normality has been introduced for a finitely complete category by Bourn in \cite{B_normal}.
Let $(R,r_1,r_2)$ be an equivalence relation on an object $A$. We call a map
$k:K\to A$ \emph{normal} to $R$ when the following two diagrams are pullbacks:
$$
\xymatrix{
K\times K \ar[r]^{\tilde{k}}\ar[d]_{id}\ar@{}[dr]|(.3){\lrcorner}
&R\ar[d]^{\langle r_1,r_2\rangle}\\
K\times K\ar[r]_{k\times k}&A\times A}
\qquad
\xymatrix{
K\times K \ar[r]^{\tilde{k}}\ar[d]_{\pi_1}\ar@{}[dr]|(.3){\lrcorner}
&R\ar[d]^{r_1}\\
K \ar[r]_{k\times k}& A.}
$$
The map $k$ results to be a mono, and diagrams above express nothing but the fact that
the subobject $K$ is an equivalence class of $R$.

Let us notice that, as $\C$ has a zero object, any equivalence relation determines a unique
normal monomorphism: it suffices to take the pullback
\begin{equation}\label{diag:B-normal}
\xymatrix{
K\ar[r]^{}\ar[d]_{k}\ar@{}[dr]|(.3){\lrcorner}
&R\ar[d]^{\langle r_1,r_2\rangle}\\
A\ar[r]_(.4){\langle 1,0\rangle }&A\times A,}
\end{equation}
i.e. in the pointed case, any normal subobject $K$ is precisely the class $[0]_R$ of
an equivalence relation $R$ (see \cite{bibbia}).

The  kernel of a map $f$ is normal to the usual kernel pair equivalence relation $R[f]$,
so that every kernel is normal. The converse is not true in general. It is so when the base category
$\C$ is moreover protomodular, with all equivalence relation effective \cite{bibbia}.

\subsection{Seminormal subobjects}
P. Agliano and A. Ursini (\cite{Agliano_Ursini}) introduced the notion of \emph{clot} for
algebraic varieties.

A subalgebra $K$ of $A$ in $\C$ is a clot in $A$ if
$$
t(a_1,\dots,a_m,0,\dots,0)=0\quad \mathrm{and}\quad k_1,\dots,k_n\in K\quad \mathrm{imply}\quad
t(a_1,\dots,a_m,k_1,\dots,k_n)\in K
$$
for every $a_1,\dots,a_m,k_1,\dots,k_n$ in $A$ and every $(m+n)$-ary term function $t$ of $A$.\\

We will deal with the (weaker) categorical notion of clot in the following section. Here we are
interested in a notion introduced \cite{JU}, which is equivalent to that of clot in the varietal case.
There, as pointed out in \cite{Agliano_Ursini}, a clot is the same as the set (subalgebra) of the elements $x\in A$
such that $0Rx$ for a given  internal reflexive relation (i.e. a semicongruence).
Hence we call a map $k:K\to A$ \emph{seminormal} in $A$ w.r.t.\ $R$ when $k$ is
obtained by a pullback as in diagram (\ref{diag:B-normal}), where $R$ is a semicongruence on $A$
(notice that in \cite{JU} the same is called clot).

Of course every normal subobject is also seminormal, as every equivalence relation is reflective.
The converse may not hold in general; however it does hold when the category $\C$ is Mal'cev.\\

In order to introduce the categorical notion of clot, we recall, for the reader's convenience the notion of \emph{internal action}.

Let $\C$ be a finitely complete pointed category with coproducts. Then for any object $B$ in $\C$
one can define a functor ``ker'' from the category of split epimorphisms (a.k.a. points) over $B$ into $\C$
$$
\mathrm{ker} : Pt_{B}(\C) \to \C,\quad \raisebox{4ex}{\xymatrix{A\ar@<2pt>@{->>}[d]^{\alpha}\\B\ar@<2pt>[u]^{\beta}}}\mapsto \mathrm{ker}(\alpha).
$$
This has a left adjoint:
$$
\mathrm{B+(-)} : \C \to Pt_{B}(\C), \quad X \mapsto \raisebox{4ex}{\xymatrix{B+X\ar@<2pt>@{->>}[d]^{[1,0]}\\B\ar@<2pt>[u]^{i_B}}},
$$
and the monad corresponding to this adjunction is denoted by
$B\flat(-)$. In fact for any object A of $\C$ one gets a kernel
diagram:
$$
\xymatrix{B\flat A\ar[r]^{n_{B,A}}&B+A\ar[r]^(.6){[1,0]}&B.}
$$
The normal monomorphism $n_{B,A}$ will be denoted simply $n$ when no confusion arises.
The $B\flat(-)$-algebras are called internal $B$-actions in $\C$ (see \cite{BJK_INT_OBJ_ACTS}).
Let us observe that in the case of groups, the object $B\flat A$ is the group generated by the \emph{formal conjugates} of elements of $A$ by elements
of $B$, i.e. by the triples of the kind $(b,a,b^{-1})$ with $b\in B$ and $a\in A$.

 For any object $A$ of $\C$, one can define a canonical action of $A$ on $A$ itself given by the composition:
$$
\chi_A:\xymatrix{A\flat A\ar[r]^{n_{A,A}}&A+A\ar[r]^(.6){[1,1]}&A}.
$$
In the category of groups, the morphism $\chi_A$ is the internal
actions associated to the usual conjugation in $A$:
 the realization morphism $[1,1]$ of above makes the
formal conjugates of $A\flat A$ computed effectively in $A$.

\subsection*{Clots}
A subobject $k:K\rightarrowtail A$ is clot in $A$,
when there exists a morphism $\xi:A\flat K\to K$ such that the diagram
$$
\xymatrix{A\flat K \ar[r]^{\xi} \ar[d]_{1\flat k}
&K\ar[d]^k
\\
A\flat A\ar[r]_{\chi_A}&A}
$$
commutes. As $k$ is a mono, the morphism $\xi$ defined above is unique: namely  it is the
internal $A\flat(-)$-action that restricts the  action in $A$.

Every seminormal subobject is closed under conjugation. In fact the restriction map $\xi$
is the internal action on the kernel corresponding to the split extension determined by
the codomain map of $R$.

 In \cite{JMU09} it is shown that, when $\C$
is a pointed regular category with finite coproducts,  also the
converse holds .

\subsection*{Ideals} The categorical notion of ideal has been recently introduced in  \cite{JMU09},
while the varietal one can be found in  \cite{Magari, U_Bit, Higg_commutator}.
A subobject $k:K\rightarrowtail A$ is an ideal in $A$
 when it is the regular image of a
clot along a regular epimorphism, i.e.\ if there exists a commutative diagram
$$
\xymatrix{
L\ar@{.>>}[r]^{f'}\ar[d]_l&K\ar[d]^k\\B\ar@{.>>}[r]_f&A.}
$$
with $K$ clot and $f$, $f'$ regular epimorphisms.
It is immediate to observe that, according to this definition, every clot is an ideal subobject.\\

The notion of ideal rests on the possibility of taking
 images, which would lead
us to set the definition more appropriately in a regular category. Let us observe that if
this is the case, since every clot is the regular image of a kernel, and since
in regular categories regular epimorphisms compose,  ideals are  regular images of
kernels (as it happens in relevant algebraic examples, see \emph{Remark 3.3} in \cite{JMU09}).
This has an interesting consequence when $\C$ is moreover Mal'cev (see \cite{CLP}).
 Bourn has shown \cite{bibbia} that in that case, regular images of kernels are normal.
Result: ideal and  normal subobjects coincide.\\

The discussion above is summarized in the following table:

\begin{center}
\begin{tabular}{|c|c|c|c|c|}
  \hline
  &&&&\\[-1ex]
  \textsc{Kernel} & \textsc{Normal} & \textsc{Seminormal} & \textsc{Clot} & \textsc{Ideal} \\
  &&&&\\[-1ex] \hline
%
%
$${\small
\xymatrix@C=3ex@R=3ex{
K\ar[r]\ar[d]_{k}\ar@{}[dr]|(.3){\lrcorner}
&0\ar[d]
\\
A\ar[r]_{f}&B}}
$$
&
%
%
$${\small
\xymatrix@C=3ex@R=3ex{
K\ar[r]\ar[d]_{k}\ar@{}[dr]|(.3){\lrcorner}
&R\ar[d]
\\
A\ar[d]_{}\ar@{}[dr]|(.3){\lrcorner}\ar[r]^{\langle1,0\rangle}&A\times A \ar[d]^{\pi_2}
\\
0\ar[r]&A}}
$$
&
%
%
$${\small
\xymatrix@C=3ex@R=3ex{
K\ar[r]\ar[d]_{k}\ar@{}[dr]|(.3){\lrcorner}
&R\ar[d]
\\
A\ar[d]_{}\ar@{}[dr]|(.3){\lrcorner}\ar[r]^{\langle1,0\rangle}&A\times A \ar[d]^{\pi_2}
\\
0\ar[r]&A}}
$$
&
%
%
$${\small
\xymatrix@C=3ex@R=3ex{
A\flat K\ar[r]^\xi \ar[d]_{1\flat k}
&K\ar[d]^k
\\
A\flat A\ar[r]_{\chi_A}&A}}
$$
&
%
%
$${\small
\xymatrix@C=3ex@R=3ex{
L\ar@{->>}[r]^{f'} \ar[d]_{l}
&K\ar[d]^k
\\
B\ar@{->>}[r]_{f}&A}}
$$
\\
&
$K=[0]_R$,
&
$K=[0]_R$,
&
$\chi_A$ factors
&
$l$ clot,
\\
&
$R$ equiv. rel.
&
$R$ refl. rel.
&
through $k$.
&
$f,f'$ regular epis.
\\&&&&\\&&&&\\&&&&\\&&&&\\&&&&\\ \hline
\end{tabular}\\
\raisebox{0pt}[0pt][0pt]{
\raisebox{+15ex}{$\xymatrix@C=16ex@R=2.5ex{\\
\ar@{}[r]|(.15){}="a1"|(.6){}="a2"\ar@{=>}"a1";"a2"&
\ar@{}[r]|(.15){}="b1"|(.6){}="b2"\ar@{=>}"b1";"b2"&
\ar@{}[r]|(.1){}="c1"|(.6){}="c2"\ar@{=>}"c1";"c2"&
\ar@{}[r]|(.0){}="d1"|(.5){}="d2"\ar@{=>}"d1";"d2"&
\\
&\ar@{}[r]|(.15){}="bb1"|(.6){}="bb2"\ar@{=>}"bb2";"bb1"^{\txt{Mal'cev}}&
\ar@{}[r]|(.1){}="cc1"|(.6){}="cc2"\ar@{=>}"cc2";"cc1"^{\txt{Regular}}&&
\\
&\ar@{}[r]|(.15){}="bbb1"&&
\ar@{}[r]|(.6){}="ddd2"\ar@{=>}"ddd2";"bbb1"^{\txt{Mal'cev + Regular}}&}$
}}
\end{center}

\section{Semi-abelian, homological and ideal-determined categories.}
Semi-abelian categories have been introduced in \cite{JMT_SEMIAB} in order to recapture
algebraic properties of groups, rings etc.\ in a categorical-theoretical
setting, just as abelian categories do for abelian groups, modules etc.

A semi-abelian category is a Barr-exact, Bourn-protomodular category with zero object and
finite coproducts. In \cite{JMT_SEMIAB} the authors compare the definition as stated
above (i.e.\ given in terms of ``new" axioms) with preexisting investigations on this
subject. This lead them to state  equivalent versions of the main definition, given
in terms of some so--called ``old" axioms, more commonly used in universal algebra.
From there we borrow the following characterization: a pointed category $\C$
with finite limits and colimits is semi-abelian if,
and only if, it satisfies A$1$, A$2$ and A$3$ below:
\begin{itemize}
\item[A$1$] $\C$ has  a pullback-stable normal-epi/mono factorization
\item[A$2$] regular images of kernels are kernels
\item[A$3$] (Hoffmann's axiom) in the diagram below, where $l$ and $l'$ are regular epimorphisms
and $x'$ is a kernel, if Ker$(l)\leq X$ then $x$ is also a kernel.
$$
\xymatrix{
X\ar@{->>}[r]^{l'}\ar[d]_{x}&X'\ar[d]^{x'}\\
Y\ar@{->>}[r]_{l}&Y'}
$$
\end{itemize}
Axiom A1 can be easily reformulated as follows:
\begin{itemize}
\item[A$1'$] $\C$ is regular and regular epi's = normal epi's.
\end{itemize}
When axiom A$1$ holds, axiom A$2$ means precisely that kernels = ideals, and
 in this case all the different notions detailed in Section \ref{sec:normalities} collapse.
Pointed categories  with finite limits and colimits satisfying
axioms A$1$ and A$2$ are called \emph{ideal-determined}   \cite{JMU09}.

Differently, as shown in Proposition 3.3.\ of \cite{JMT_SEMIAB}, axioms
A$1$ and A$3$ characterize homological categories among those with
finite limits and zero object,  where homological means pointed,
regular and protomodular, according to the definition due to
Borceux and Bourn  \cite{bibbia}.

The formulation of the notion of semi-abelian category in terms of
the (old) axioms A$1$, A$2$ and A$3$ led us to the following
analysis (in the perspective of the observation in 2.7 of \cite{JMT_SEMIAB}).

Let us consider the diagram:
$$
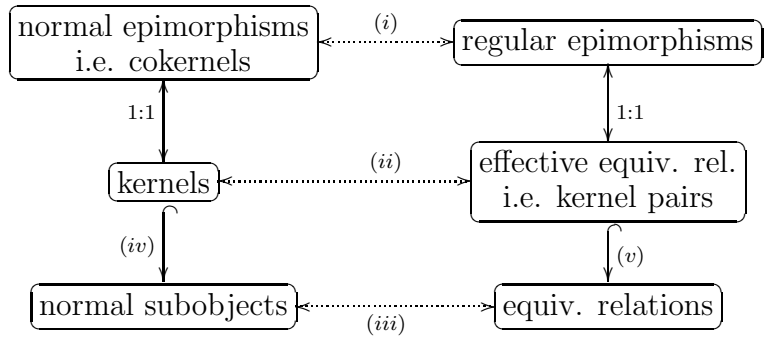
\xymatrix@C=10ex{
*+[F:<3pt>]\txt{normal epimorphisms\\ i.e. cokernels}
\ar@{<.>}[r]^{(i)}
\ar@{<->}[d]_{1:1}
&
*+[F:<3pt>]\txt{regular epimorphisms}
\ar@{<->}[d]^{1:1}
\\
*+[F:<3pt>]\txt{kernels}
\ar@{<.>}[r]^{(ii)}
\ar@{}[d]|(.25){}="b1"|(.8){}="b2"
\ar@{^(->}"b1";"b2"_{(iv)}
&
*+[F:<3pt>]\txt{effective equiv. rel.\\ i.e. kernel pairs}
\ar@{}[d]|(.4){}="a1"|(.8){}="a2"
\ar@{^(->}"a1";"a2"^{(v)}
\\
*+[F:<3pt>]\txt{normal subobjects}
\ar@{<.>}[r]_{(iii)}
&
*+[F:<3pt>]\txt{equiv. relations}
}$$
If they are considered separately, the two columns describe quite general facts which hold in categories
where referred items do exists: in a category with pullbacks and coequalizers, regular epimorphisms
are in one-to-one correspondence with kernel pairs, and any kernel pair is an equivalence relation.
Similarly, in a category with kernels and cokernels one gets a one-to-one correspondence between them, and any kernel is
a normal subobject.

Now, axiom A$1$, in the form of A$1'$, amounts to $\C$ being a
regular category, plus the arrow $(i)$ being an equality. This
implies immediately that the unique map $(ii)$, that makes the
upper square commute, is a bijection, and hence the following
valuable fact holds: any effective equivalence relation is
completely determined by its zero-class.

When also axiom A$3$ holds, i.e.\ when the category $\C$ is
homological, this fact can be extended to equivalence relations,
that are univocally determined by their zero classes \cite{bibbia}.
 In this case, the bijection $(ii)$ is just the restriction
of a more general bijection $(iii)$.

Finally let us consider the new axiom:
\begin{itemize}
\item[A$2'$] Kernels = normal subobjects.
\end{itemize}
This means precisely that the inclusion $(iv)$ is indeed an equality. Hence, when the category $\C$ is homological,
this fact collapses the lower square of the diagram, so that  axiom A$2'$ is equivalent to the inclusion $(v)$
being an equality. This means that  the regular category $\C$ is moreover exact. \\

Considerations above give the following characterization:

\begin{Corollary}
A category $\C$ is semi-abelian if, and only if, axioms A$1'$, A$2'$ and A$3$ hold.
\end{Corollary}

\section{Commutators}
Several notions of commutators of two coterminal (i.e.\ with the same codomain) morphisms
have been proposed and studied in different algebraic contexts.

In this section we explore some aspects of commutator theory, by comparing categorically
 the two different notions of commutator
introduced by Higgins \cite{Higg_commutator} and Huq \cite{Huq68}. Our approach rests on the observation that
both these notions can be formulated  from the same notion of \emph{formal commutator}, but first we need
to recall the original definitions.\\

\subsection{Higgins' commutator} The notion of commutator according to Higgins,
generalizes  the case of groups  to those of
$\Omega$-groups \cite{Higg_commutator}.

We recall the definition of $\Omega$-groups for the reader's convenience.
A category $\mathcal{V}$ of $\Omega$-groups  is a variety of
groups (in the sense of the universal algebra) such that:
\begin{itemize}
\item[-] the group identity is the only operation of arity $0$,
i.e. the variety is pointed; \item[-] all other operations
different from group operation (here written additively), inverse
and identity, have arity $n,$ with  $n\geq 1$.
\end{itemize}

Given a $\Omega$-group
$A$ and two subobjects $h\!:H\rightarrowtail A$ and $k\!:K\rightarrowtail A$,
the Higgins' commutator $[H,K]_H$ is set of all
$f(\overrightarrow{h},\overrightarrow{k})$ with $\overrightarrow{h}\in H^n$
and $\overrightarrow{k}\in K^m$, $f$ being commutator words, i.e. such that
$f(\overrightarrow{0},\overrightarrow{y})=0=f(\overrightarrow{x},\overrightarrow{0})$.\\

Actually it is possible to describe Higgins' commutator as the
 image of a certain subobject of the coproduct $H+K$
through  the canonical map $[h,k]:H+K\to A$. This can be given in
an intrinsic way, which will make a generalization
easier.

Let us fix some notation.
Let $H$ and $K$ be two given objects of a pointed finitely
complete category $\C$ with coproducts.
We denote by $\Sigma_{H,K}$ (or simply $\Sigma$) the canonical arrow
$$
\Sigma_{H,K}=\langle[1,0],[01]\rangle = [\langle1,0\rangle,\langle0,1\rangle]: H+K\to H\times K.
$$
Let us observe that $\Sigma_{H,K}$ is a regular epimorphism for every pair of objects $H$ and $K$
in $\C$ if, and only if, $\C$ is unital \cite{bibbia}.
\subsection{Formal commutator}
Let us consider two objects $H$ and $K$ of $\C$.
We define the
\emph{formal commutator} of $H$ and $K$ as the kernel $(H\diamond
K,\sigma_{H,K})$ of $\Sigma_{H,K}$.
 Actually $H\diamond K$ is the
intersection of
 $H\flat K$ and $K\flat K$, as shown in the diagram below:
$$
\xymatrix{
H\diamond K \ar[r]^{}\ar[d]_{}\ar@{.>}[dr]|{\sigma_{H,K}}
&H\flat K\ar[d]
&
\\
K\flat H\ar[r]
&H+K\ar[r]^{[0,1]}\ar[d]_{[1,0]}\ar@{.>}[dr]|{\Sigma_{H,K}}&
K\ar@{<-}[d]\\
&H\ar@{<-}[r]
&H\times K.}
$$
The reason why we call $H\diamond K$ \emph{formal commutator} is,
from one side, that it is just the  Huq's commutator (that we are
going to define later) of $H$ and $K$ in  $H+K$ (see Remark
\ref{rem:formal_comm}).

Our motivation for using the therm \emph{formal} originates by the fact that,
 in the category of groups, the elements of $H+K$ can be
represented as  reduced formal juxtapositions of elements of $H$
and $K$, say sequences of the kind $(h_1,k_1,\cdots,h_n,k_n)$.
 $H\diamond K$ is generated by all the words of the kind $(x,y,x^{-1},y^{-1})$,
with $x\in H$ and $y\in K$.

\subsection{The realization map}\label{realization}
The canonical map
$$
[h,k]:H+K\to A
$$
has an interesting interpretation when the object $A$ is a group
(actually, similar arguments hold in any  pointed variety of
universal algebra). The map $[h,k]$ acts on sequences of the kind
$(h_1,k_1,\cdots,h_n,k_n)$ by means of the group operation
of $A$, thus giving the element   $h_1 k_1 \cdots h_n k_n$ \emph{computed} in $A$.\\

Now, if $(H,h)$ and $(K,k)$ are  subgroups of a given group $A$,
one  can easily check that the image of $H\diamond  K$ through $[h,k]$ is
precisely the commutator subgroup of $H$ and $K$ in $A$.\\

At this point one can  easily define the internal version of
Higgins' commutator in an ideal-determined category.
\begin{Definition}Let $\C$ be ideal-determined. Then the Higgins' commutator $[H,K]_H$
of two subobjects $h:H\rightarrowtail A$ and $k:K\rightarrowtail A$ of $A$
is the regular image of $H\diamond K$ under the morphism $[h,k]\sigma_{H,K}$  (see diagram below)
$$\newdir{|>}{%
!/4.5pt/@{|}*:(1,-.2)@^{>}*:(1,+.2)@_{>}}
\xymatrix{*+<2ex>{H\diamond K}\ar@{->>}[r]\ar@{|>->}[d]_{\delta_{H,K} }
&*+<2ex>{[H,K]_H} \ar@{>->}[d]
\\
H+K\ar[r]_{[h,k]}&A.}
$$
\end{Definition}

As far as the ground category $\C$ has a regular epi/mono
factorization system, as in the ideal-determined case, we can
obtain the join of two subobjects $(H,h)$ and $(K,k)$ of a given
object $A$ as the regular image of the canonical map $[h,k]$:
$$
\xymatrix{
&H\vee K\ar[dr]^q\\
H+K\ar@{->>}[ur]^p\ar[rr]_{[h,k]}&&A.}
$$
Then, the following proposition holds.
\begin{Proposition}Let $\C$ be ideal-determined. Then the Higgins' commutator $[H,K]_H$
of two subobjects $h:H\rightarrowtail A$ and $k:K\rightarrowtail A$ of $A$
is a normal subobject of $H\vee K$.
\end{Proposition}
\begin{proof}
The inclusion of $[H,K]_H$ in $A$ factors through $H\vee K$ by a
monomorphism, which is normal as a consequence of axiom A$2$; see
diagram below:
$$\newdir{|>}{%
!/4.5pt/@{|}*:(1,-.2)@^{>}*:(1,+.2)@_{>}}
\xymatrix{*+<2ex>{H\diamond K}\ar@{->>}[r]\ar@{|>->}[d]_{\delta_{H,K} }
&*+<2ex>{[H,K]_H} \ar@{|>.>}[d] \ar@{>->}[dr]
\\
H+K\ar@{->>}[r]\ar@/_4ex/[rr]_{[h,k]}&H\vee K
\ar@{>->}[r]&A.}
$$
\end{proof}

\subsection{Huq's commutator}
The construction described below  was introduced by S. A. Huq \cite{Huq68} in a purely categorical setting,
and further developed by D. Bourn \cite{Bourn_commutator}. We borrow the general definition from the second author,
but we will restrict our attention to the cases when the two coterminal morphisms are monomorphisms.

\begin{Definition}\label{def:commutator_Bourn}
Let $\C$ be a finitely complete, unital category,  such that regular and normal epimorphisms
coincide, and let us consider two subobjects $h\!:H\rightarrowtail A$ and $k\!:K\rightarrowtail A$.
We define the commutator quotient $Q=Q(h,k)$ as the colimit of the solid arrows in the diagram below,
if it exists:
\begin{equation}\label{diag:diamond}
\newdir{|>}{%
!/4.5pt/@{|}*:(1,-.2)@^{>}*:(1,+.2)@_{>}}
\xymatrix@!=7ex{
&H\ar[dl]_{\langle1,0\rangle}\ar@{.>}[d]\ar[dr]^{h}\\
H\times K\ar@{.>}[r]^m
&Q&A\ar@{.>>}[l]_{q}&*+<2ex>{[H,K]_Q}\ar@{|>.>}[l]_(.6)j\\
&K\ar[ul]^{\langle0,1\rangle}
\ar@{.>}[u]\ar[ur]_{k}
}
\end{equation}
The kernel of $q$ is denoted by $[H,K]_Q$ and we will refer to it as to the Huq's commutator of $H$ and $K$.
\end{Definition}
Let us observe that when the map $q$ is an isomorphism, this means
precisely that  $h$ and $k$ cooperate, i.e.\ $H$ and $K$ commute in
$A$. Hence the distance for $q$ from being an isomorphism
expresses exactly the lack of commutativity between $H$ and $K$.
Since $q$ is a normal  epimorphism,  this can be measured by its
kernel $[H,K]_Q$.

\begin{remark}
All that follows is indeed well known, but it is worth recalling it, as
it describes the commutator (quotient) as a kind of universal representer
of the algebraic operations.

Let $\C$ be the category of groups.
The group operation $\cdot\!:\xymatrix{H\times K\ar@{-->}[r]&A}$ is not a morphism in general.
Nevertheless one can quotient the group $A$ with some normal subgroup, such that the quotient map is indeed a morphism.
This is always possible: a trivial answer is to quotient $A$ with itself. A better answer is given by the commutator of
$H$ and $K$. This is in fact the smaller normal subgroup such that such that the quotient map is a morphism: with
$Q=A/[H,K]_Q$ one has:
$$
\xymatrix@C=6ex{
H\times K\ar@/_3ex/[rr]_m\ar@{-->}[r]^(.6){\cdot}&A\ar[r]^(.4){q}&A/[H,K]_Q.
}
$$
Finally, if $H=K=A$ we get the  Heckmann-Hilton argument: a group $A$ is abelian
if, and only if, the group operation $A\times A\to A$ is a morphism, and this happens
precisely when the derived subgroup $[A,A]_Q$ is trivial.
\end{remark}

We observe  that  the definition of the commutator quotient  can
be expressed by means of a canonical pushout, as one can easily
see in the next  proposition:
\begin{Proposition} Let $\C$ be as before, with finite sums.
The colimit diagram (\ref{diag:diamond}) is equivalent to the following pushout
\begin{equation}\label{diag:pushout}
\xymatrix{
H+K\ar[r]^{[h,k]}
\ar[d]_{\Sigma}
&A\ar[d]^q\\
H\times K\ar[r]_m
&Q\ar@{}[ul]|(.2){\ulcorner}
}.
\end{equation}
\end{Proposition}

In this context, the Huq's commutator $[H,K]_Q$ of $H$ and $K$ is
obtained as the kernel of the map $q$ of diagram
(\ref{diag:pushout}).
\begin{remark}\label{rem:formal_comm}
It is worth considering the very special case when $A=H+K$, and
$h$ and $k$ are the canonical inclusions. In this situation
diagram (\ref{diag:pushout}) trivializes, and the Huq's commutator
$[H,K]_Q$ of $H$ and $K$ in $H+K$ coincides with the formal
commutator $H\diamond K$, whence its name.
\end{remark}

\subsection{Higgins commutator and Huq commutator compared}

The discussion above suggests to develop some considerations. Let
$\C$ be an    ideal-determined, unital category and let $H,K$ be
subobjects of an object $A$ as above. This is a nice context where
Huq's commutator and Higgins' commutator can be compared.

In general these two notions do not coincide, more precisely
 $[H,K]_H$ is a proper  subobject of $[H,K]_Q$, since $q$
 restricted to  $[H,K]_H$ is zero, as the following diagram shows:
$$\newdir{|>}{%
!/4.5pt/@{|}*:(1,-.2)@^{>}*:(1,+.2)@_{>}}
\newdir{ >}{{}*!/-5pt/@{>}}
\xymatrix{
&
[H,K]_H \ar@{|>->}[dd]\ar@{ >->}[dr]
\\
*+<2ex>{H\diamond K} \ar@{->}[ur]\ar@{|>->}[dd]_{\sigma}\ar[rr]
&&[H,K]_Q\ar@{|>->}[dd]
\\
&*+<2ex>{H\vee K} \ar@{->>}[dd] \ar@{ >->}[dr]\ar@{}[dddr]|(.9){\ulcorner}
\\
H+K\ar@{->>}[dd]_{\Sigma}\ar@{->>}[ur]\ar[rr]_(.65){[h,k]}\ar@{}[dr]|(.85){\ulcorner}
&&A\ar@{->>}[dd]^q
\\
&\bullet\ar[dr]
\\
H\times K\ar@{->>}[ur] \ar[rr]
&&\bullet
}
$$

Let us observe that in the above diagram, the three vertical
sequence of morphisms are exact and the two bottom \lq\lq
diamonds" are pushouts.\medskip

To be more precise, the relationship between Huq's commutator and
Higgins' commutator is explained by the following proposition.
\begin{Proposition}
Let $\C$ be an ideal-determined, unital category, and let $H$ and
$K$ be subobjects of  $A$. Then
$$
[H,K]_Q=\overline{[H,K]}_H,
$$
i.e. the Huq's commutator is the normalization in $A$ of the Higgins' commutator.\end{Proposition}
\begin{proof}
Since $\C$ is unital, $\Sigma$ is a regular epi, and so is  $q$.
Actually both are cokernels by axiom A$1$, and
 a direct calculation shows that $q$ is precisely the cokernel of the inclusion $[H,K]_H\rightarrowtail A$.
\end{proof}

A natural question  to ask at this point is under what conditions on $\C$ the comparison
is an isomorphism, so that the two notions coincide. As we already
noticed, this is not true in general, and even for the category of
groups the two notions are distinct, if $H$ and $K$ are not
normal, as the following example shows.

\begin{Example}\label{Example:Higg_not_Huq}
Let us consider the simple group $A_5$, given by even permutations of order five, and the two subgroups
$H=\langle (12)(34) \rangle$ and $K=\langle (12)(45)\rangle$. Then $[H,K]_H=\langle (345)\rangle\neq[H,K]_Q=A_5$
(see \cite{Alan} for a detailed discussion).
\end{Example}

Of course, Huq's commutator and Higgins' commutator coincide
 when the subobjects are sufficiently big, i.e.\ when $H\vee K =A$
 (as, for example, in the case of the formal commutator, where $H\vee
 K=H+K$).
In particular, this happens when one of the subobjects, say $H$,
is the whole $A$; then the map $[1,k]:A+K\to A$ is a regular epimorphism,
and the diagram above happily collapses. If this is the case, we
will drop the $H$ and the $Q$ subscripts, and write simply $[A,K]$
for \emph{the} commutator of $A$ and $K$.
Another reason why the case $H=A$ is special is that in this case
Huq's commutator behaves well w.r.t. normalization, as shown in the Proposition \ref{kbar} below.
In general this is not true, as one can verify for the simple group $A_6$,
with $H=\langle (123) \rangle$ and $K=\langle (456) \rangle$, where $[H,K]_Q=0$, while
$[\overline{H},\K]_Q=A_6$ (see \cite{Alan}).

\begin{Proposition}\label{kbar}
Let $\C$ be an ideal-determined unital category. Then for a
monomorphism $k:K\rightarrowtail A$ one has
$$
[A,K]=[A,\K],
$$
where $(\overline{K},\bar{k})$ is the normalization of $(K,k)$.
\end{Proposition}
\begin{proof}
Let us refer to the two diagrams below:
$$
\newdir{|>}{%
!/4.5pt/@{|}*:(1,-.2)@^{>}*:(1,+.2)@_{>}}
\newdir{ >}{{}*!/-5pt/@{>}}\xymatrix@C=7ex{
*+<2ex>{A\diamond(A\flat K)}\ar@{}[dr]|{(i)}\ar[r]^q\ar@{|>->}[d]_{D}
&*+<2ex>{A\diamond K}\ar@{}[dr]|{(ii)}\ar[r]^{p}\ar@{|>->}[d]^d
&*+<2ex>{[A,K]}\ar@{|>->}[d]^c
\\
A\flat(A\flat K)\ar[r]_{(\chi_{A}\flat A)\varphi}
&A\flat K\ar@{->>}[r]_{\chi}
&\K,}
$$

$$
\newdir{|>}{%
!/4.5pt/@{|}*:(1,-.2)@^{>}*:(1,+.2)@_{>}}
\newdir{ >}{{}*!/-5pt/@{>}}\xymatrix@C=7ex{
*+<2ex>{A\diamond(A\flat K)}\ar@{}[dr]|{(iii)}\ar[r]^{\overline{q}}\ar@{|>->}[d]_{D}
&*+<2ex>{A\diamond \K}\ar@{}[dr]|{(iv)}\ar[r]^{\overline{p}}\ar@{|>->}[d]^{\overline{d}}
&*+<2ex>{[A,\K]}\ar@{|>->}[d]^{\overline{c}}
\\
A\flat(A\flat K)\ar[r]_{A\flat \chi}
&A\flat \K\ar@{->>}[r]_{\overline{\chi}}
&\K.}
$$
The idea of the proof is to show that the compositions $pq$ and $\bar{p}\bar{q}$ are
the regular images of the same (to be proved) morphism
$\chi (\chi_{A}\flat K)\varphi D = \overline{\chi}(A\flat \chi) D$.
We shall start by defining all the characters playing in the diagrams.
\begin{itemize}
\item The arrows $D$, $d$ and $\overline{d}$ are  kernels:
$$\newdir{|>}{%
!/4.5pt/@{|}*:(1,-.2)@^{>}*:(1,+.2)@_{>}}
\xymatrix@C=7ex{A\diamond(A\flat K)\ar@{|>->}[r]^D&A\flat(A\flat K)\ar[r]^{[0,1]n}&A\flat K}
$$
$$
\newdir{|>}{%
!/4.5pt/@{|}*:(1,-.2)@^{>}*:(1,+.2)@_{>}}
\xymatrix@C=5ex{A\diamond K\ar@{|>->}[r]^d&A\flat K\ar[r]^{[0,1]n}& K,}\quad
\xymatrix@C=5ex{A\diamond \overline{K}\ar@{|>->}[r]^{\overline{d}} &A\flat \overline{K}\ar[r]^{[0,1]n}& \overline{K}.}
$$
\item The arrows $\chi$ and $\overline{\chi}$ are the regular images of the compositions:
$$
\newdir{|>}{%
!/4.5pt/@{|}*:(1,-.2)@^{>}*:(1,+.2)@_{>}}
\xymatrix@C=5ex{A\flat K\ar@{|>->}[r]^n&A+ K\ar[r]^{[1,k]}& K,}\quad
\xymatrix@C=5ex{A\flat \overline{K}\ar@{|>->}[r]^{n} &A+ \overline{K}\ar[r]^{[1,k]}& \overline{K}.}
$$
\item In diagrams $(ii)$ and $(iv)$, up-right compositions are regular epimorphisms followed
by monomorphisms, factorizing left-down compositions.
\item The arrow $\varphi$ is given by universal property of kernels, see the diagram below,
where  the lower part commutes:
$$
\newdir{|>}{%
!/4.5pt/@{|}*:(1,-.2)@^{>}*:(1,+.2)@_{>}}
\newdir{ >}{{}*!/-5pt/@{>}}
\xymatrix@C=7ex{
A\flat(A\flat K)\ar@{.>}[rr]^{\varphi}\ar@{|>->}[d]_{n}
&&(A\flat A)\flat K\ar@{|>->}[d]^{n}\\
A+(A\flat K)\ar[d]_{[1,0]}\ar[r]^{A+n}
&A+A+K\ar[r]^{[\eta_A,\eta_A]+K}
&(A\flat A)+ K\ar[d]^{n[1,0]}
\\
A\ar[rr]_{\eta_A}&&A\flat A.
}$$
In fact:
$$
[1,0] ([\eta_A,\eta_A]+K)(A+n)
=[1,0] (\eta_A K) ([1,1]+K)(A+n)
=\eta_A [1,0]  ([1,1]+K)(A+n)=
$$
$$
=\eta_A [1,1]  (A+[1,0])(A+n)
=\eta_A [1,1]  i_1[1,0]
=\eta_A [1,0].
$$

\item The arrow $q$ is given again by universal property of kernels, so that diagram $(i)$
commutes for free:
$$
\newdir{|>}{%
!/4.5pt/@{|}*:(1,-.2)@^{>}*:(1,+.2)@_{>}}
\newdir{ >}{{}*!/-5pt/@{>}}
\xymatrix@C=6ex{
A\diamond(A\flat K)\ar@{|>->}[d]_{D}\ar@{.>}[rrr]^q&&&A\diamond K\ar@{|>->}[d]^{d}
\\
A\flat(A\flat K)\ar[rr]^{\varphi}\ar@{|>->}[d]_{n}
&&(A\flat A)\flat K\ar@{|>->}[d]^{n}\ar[r]^{\chi_A\flat K}
&(A\flat K)\ar@{|>->}[d]^{n}
\\
A+(A\flat K)\ar[d]_{[0,1]}\ar[r]^{A+n}
&A+A+K\ar[r]^{[\eta_A,\eta_A]+K}\ar[d]^{[0,1_{A+K}]}
&(A\flat A)+ K\ar[d]^{[0,0,1]}
\ar[r]^{\chi_A+ K}
&(A+ K)\ar[d]^{[0,1]}\\
A\flat K\ar[r]_n
&A+K\ar[r]_{[0,1]}
&K\ar[r]_{Id}
&K.
}$$
In fact:
$$
[1,0]n(\chi_A\flat K)\varphi
=[1,0](\chi_A+ K)n\varphi
=[1,0](\chi_A+ K)([\eta_A,\eta_A]+K)(A+n)n=
$$
$$
=[0,0,1](n+ K)([\eta_A,\eta_A]+K)(A+n)n
=[0,0,1]([i_2,i_2]+K)(A+n)n=
$$
$$=[0,0,1](A+n)n
=[0,1]n[0,1]n.
$$
Moreover $q$ is a split epimorphism, and hence regular. This can be shown by precomposing
the upper part of the diagram above  with the monomorphism $A\diamond \eta_K$.
The calculation shows that $ndq(A\diamond \eta_K)=nd$, which is a monomorphism.
Canceling $nd$, one gets $q(A\diamond \eta_K)=1_{A\diamond K}$.
\item The arrow $\bar{q}$ is simply $A\diamond \chi$, so that diagram $(iii)$ commutes. Moreover
$\bar{q}$  is a regular epimorphism by
Lemma \ref{epireg}.
\end{itemize}
The discussion above has shown that diagrams $(i)$, $(ii)$, $(iii)$ and $(iv)$ commute, and
that their upper sides are regular epimorphisms. Moreover they are followed by monomorphisms. Now we are to
show that their lower sides coincide: uniqueness of the factorization will conclude the proof.

Indeed,  it is convenient to compose the morphisms that we want to prove equal,
 with the normal monomorphism
$\bar{k}$.
On one side one has: $\bar{k} \overline{\chi} A\flat \chi= [1,\bar{k}](A+\chi)n$.
The other can  be represented by the commutativity of the
diagrams below:
$$
\xymatrix@C=7ex{
A\flat(A\flat K)\ar[rr]^{(\chi_A\flat K)\varphi}\ar[d]_n
&&A\flat K\ar[r]^{\chi}\ar[d]^n
&\overline{K}\ar[d]^{\bar{k}}\\
A+(A\flat K)\ar[r]^{A+n}\ar[ddr]_{A+\chi}
&A+A+K\ar[r]^{[1,1]+k}\ar[d]_{A+[1,K]}
&A+K\ar[r]^{[1,k]}
&A
\\
&A+A\ar[urr]_{[1,1]}
\\
&A+\overline{K}\ar[u]^{A+\bar{k}}\ar[uurr]_{1,\bar{k}}
}
$$

\end{proof}

\begin{Lemma}\label{epireg}
Let the category $\C$ be ideal-determined, and let $X$ be an object of $\C$. Then
\begin{itemize}
\item the functors $(-)\flat X$ and $X\flat (-)$ preserves the regular epimorphisms.
\item the functors $(-)\diamond X$ and $X\diamond (-)$ preserves the regular epimorphisms.
\end{itemize}
\end{Lemma}
\begin{proof}
Let us consider the following two diagrams:
$$
\newdir{|>}{%
!/4.5pt/@{|}*:(1,-.2)@^{>}*:(1,+.2)@_{>}}
\newdir{ >}{{}*!/-5pt/@{>}}
\xymatrix{
*+<2ex>{X\flat A}\ar@{.>>}[r]^{m}\ar@{|>->}[d]_{n}\ar@{}[dr]|{(i)}
&*+<2ex>{Z}\ar@{>.>}[d]
\\
X+A\ar[r]^{X+f}\ar[d]_{[1,0]}\ar@{}[dr]|{(ii)}
&X+B\ar[d]^{[1,0]}\\
X\ar[r]_{1}&X}\qquad
\xymatrix{
*+<2ex>{A\flat X}\ar@{.>>}[r]^{m'}\ar@{|>->}[d]_{n}\ar@{}[dr]|{(iii)}
&*+<2ex>{Z'}\ar@{>.>}[d]
\\
A+X\ar[r]^{f+X}\ar[d]_{[1,0]}\ar@{}[dr]|{(iv)}
&B+X\ar[d]^{[1,0]}\\
A\ar[r]_{f}&B}
$$
The square diagrams $(ii)$ and $(iv)$ are pushout: the first because the arrow $f+X$
is a regular epimorphism since $f$ and $1_X$ are, the second for the fact that
the outer and the left-most rectangles below are cokernels:
$$
\newdir{|>}{%
!/4.5pt/@{|}*:(1,-.2)@^{>}*:(1,+.2)@_{>}}
\newdir{ >}{{}*!/-5pt/@{>}}
\xymatrix{
X\ar[r]^{i_X}\ar[d]&A+X\ar[r]^{f+X}\ar[d]_{[1,0]}\ar@{}[dr]|{}
&B+X\ar[d]^{[1,0]}\\
0\ar[r]&A\ar[r]_{f}&B.}
$$
Hence one can take regular images $(i)$ and $(iii)$; since the category is ideal-determined, they
will coincide with the kernels of the pushout:
$$
Z=X\flat B,\quad m= X\flat f; \qquad Z'= B\flat X, \quad m'= X\flat f.
$$
This completes the proof of the first part of the lemma.\\

In order to prove the second part of the lemma, we will point our attention to the functor
$(-)\diamond X$, as the two are naturally isomorphic.

The argument of the proof is similar to that of the first part. Let us consider the
diagram:
$$
\newdir{|>}{%
!/4.5pt/@{|}*:(1,-.2)@^{>}*:(1,+.2)@_{>}}
\newdir{ >}{{}*!/-5pt/@{>}}
\xymatrix{
*+<2ex>{X\diamond A}\ar@{.>>}[r]^{m}\ar@{|>->}[d]_{d}\ar@{}[dr]|{}
&*+<2ex>{Z}\ar@{>.>}[d]
\\
X \flat A\ar[r]^{X\flat f}\ar[d]_{ n}\ar@{}[dr]|{(v)}
&X\flat B\ar[d]^{n}\\
X+A\ar@{}[dr]|{(vi)}\ar[r]^{X+f}\ar[d]_{[0,1]}&X+B\ar[d]^{[0,1]}
\\
A\ar[r]_{f}&B.}
$$
Here, again, $(m,Z)$ is the regular image of $(X\flat f)n$, so that it suffices to show that $(v)+(vi)$
is a pushout in order to get the result.
The proof that $(iv)$ is a pushout is essentially the same we saw for $(iv)$, while for $(v)$ some
more calculations are needed.

Let two converging arrow $\alpha$ and $\beta$ be given, such that $\beta n = \alpha(X\flat f)$.
Then, by universal property of the coproduct $X+B$ we get a unique $\gamma$ such that
$\alpha \eta_B=\gamma n \eta_B$ and $\beta i_X= \gamma(X+f) i_X$, where $i_X:X\to X+A$ is the
coproduct injection.

Claim:  $\beta = \gamma(X+f)$ and $\alpha= \gamma n$.
The first claim is obtained by precomposing with coproduct injections into $A+X$.
The first one, $\beta i_X= \gamma(X+f) i_X$, is given above. For the second one, just follow
the chain of equalities:
$$
\beta i_a= \beta n \eta_A = \alpha (X\flat f) \eta_A = \alpha \eta_B f= \gamma n \eta_B f=
\gamma (X+f) n \eta_A = \gamma (X+f) i_A,
$$
where the second equality holds by hypothesis, the fourth is given above and the other by definition.
Finally we have to prove the second claim, but it is a consequence of the first one, which
justifies the middle equality in the chain below:
$$
\gamma(X\flat f) = \gamma (X + f)n = \beta n= \alpha (X\flat  f).
$$
\end{proof}
\section{(Yet) another notion of normality}
In the category of groups, one can characterize normal subgroups in many different equivalent ways. The one
we present below is surely not one of the best known. Nevertheless it has interesting implications:
 it establishes a link between the notion of normal subgroup
with that of commutator. Moreover it extends unexpectedly to the semi-abelian setting.\\

\subsection{The case of groups}
 Let  $K$ be a subgroup of a given
(multiplicative) group $A$. Then $K$ is normal in $G$ if, and only
if, $[A,K]\leq K$. {\em In fact, if $K$ is normal in $A$, for
every pair of elements $k\in K$ and $a\in A$, $ak^{-1}a^{-1}\in
K$. Hence also $kak^{-1}a^{-1}\in K$, i.e.\ all generators of
$[A,K]$ are in $K$. Conversely, whenever $[A,K]\leq K$, for every
pair of elements $k\in K$ and $a\in A$ one has
$y=k^{-1}aka^{-1}\in K$, so that $ky=aka^{-1}\in K$, i.e.\ $K$ is normal in $A$. }\\

One implication of the characterization given above holds in a quite general setting. This
is established  by  the following:

\begin{Proposition}\label{neces}
Let $\C$ be an ideal-determined unital category. If $K$ is a
normal subobject of $A$, then  $[A,K]$ is a subobject of $K$.
\end{Proposition}
\begin{proof}
Let us refer to the diagram below:
$$\newdir{|>}{%
!/4.5pt/@{|}*:(1,-.2)@^{>}*:(1,+.2)@_{>}} \xymatrix@C=10ex{
*+<2ex>{A\diamond K}\ar@{->>}[r]\ar@{|>->}[d]_{\delta_{A,K} }
&*+<2ex>{[A,K]}
\ar@{|>->}[d]^j\ar@{>.>}[r]&*+<2ex>{K}\ar@{|>->}[dl]^k
\\
A+K\ar@{}[dr]|{(\bullet)}\ar@{->>}[r]^{[1,k]}\ar@{->>}[d]^{\Sigma}\ar@/_5ex/@{->>}[dd]_{[1,0]}&A\ar@{->>}[d]^q \ar@/^5ex/@{->>}[dd]^p
\\
A\times
K\ar@{}[dr]|{}\ar@{->>}[r]\ar@{->>}[d]^{\pi_A}&\frac{A}{[A,K]}\ar@{->>}[d]
\\
A\ar@{->>}[r]_p&C.
}
$$
Firstly we consider the pushout of $[1,0]$ along $[1,k]$.
Precomposition with the injection $A\rightarrowtail A+K$ forces
the other two morphisms of the pushout to be equal. Let's call
them $p\!:A\to C$.

We claim that $(C,p)$ is the cokernel of $k$. Actually, if we
denote by $i_K$ the canonical inclusion of $K$ into $A+K$, one can
compute $pk=p[1,k]i_K=p[1,0]i_K=p0=0$. Moreover, for any other
morphism such that $fk=0$, one easily shows that $f[1,k]=f[1,0]$.
Universality of pushouts gives the morphism from $C$ stating  that
$C=A/K$.

But $k$ is a kernel, so that $k=ker(p)$. Hence, in order to prove
that $[K,A]\leq K$ it suffices to show that $pj=0$. This is done
by factoring $p$ by $q=coker(j)$,
 since the square $(\bullet)$ is a pushout by definition.
 \end{proof}

In order to get the full characterization of a normal subobject
in terms of its commutator subobject, we shall move to the semi-abelian setting.
This is done in Theorem \ref{THM} below, but first we need to present the
following quite general:

\begin{Lemma}\label{alfa}
Let $\C$ be a unital category. Given a morphism $\alpha: A\times K\to A\times N $ such that
$$\pi_{A}\alpha=\pi_{A}, \qquad \alpha \langle 1,0\rangle=\langle 1,0\rangle,$$
then $\alpha=1\times r$, where $r=\pi_{N}\alpha\langle 0,1\rangle$.
Furthermore, $r$ is a regular epimorphism, if $\alpha$ is.
\end{Lemma}
 \begin{proof}
 Since  $\C$ is unital, $\alpha=1\times r$ if and only if the equality holds when precomposing with the canonical injections into the product, that is
 $$ (i) \quad\alpha \langle 1,0\rangle=(1\times r)\langle 1,0\rangle,$$
  $$ (ii) \quad\alpha \langle 0,1\rangle=(1\times r)\langle 0,1\rangle.$$
Since  $\alpha$ and $1\times r$ are morphisms to a product, it is sufficient to test the equalities above when composing  with the projections. But
   $$ (i) \,\,\pi_{_{A}}\alpha \langle 1,0\rangle=\pi_{_{A}}\langle 1,0\rangle=1= \pi_{_{A}}(1\times r)\langle 1,0\rangle, \quad
     \pi_{_{N}}\alpha \langle 1,0\rangle=\pi_{_{N}}\langle 1,0\rangle=0= \pi_{_{N}}(1\times r)\langle 1,0\rangle $$
     and
     $$ (ii) \,\,\pi_{_{A}}\alpha \langle 0,1\rangle=\pi_{_{A}}\langle 0,1\rangle=0= \pi_{_{A}}(1\times r)\langle 1,0\rangle, \quad
     \pi_{_{N}}\alpha \langle 1,0\rangle=r= \pi_{_{N}}(1\times r)\langle 1,0\rangle. $$
     Consequently $\alpha=1\times r.$
     Furthermore,  $\pi_{_{N}}\alpha=\pi_{_{N}}(1\times r)=r\pi_{_{K}}$, hence, when $\alpha$ is a regular epimorphism,
      so is $r.$
 \end{proof}

\begin{Theorem}\label{THM}
Let $\C$ be a semi-abelian category, and $K\rightarrowtail A$. Then
$K$ is a normal subobject of $A$ if, and only if, $[A,K]$ is a subobject of $K$.
\end{Theorem}
 \begin{proof}
  Since any semi-abelian category is both ideal-determined and unital, the necessary condition is given by the Proposition \ref{neces}.
  So we have to prove that, given
  $[A,K]\leq K\leq A,$ $K$ is  normal in $A$.

  If $p:A\to C$ is the cokernel of $k$ and $\overline{k}:\K\to A$ is the kernel of $p$,
  then $p=$coker($\bar{k}$).  As we have seen in the Proposition \ref{neces}, $p$ can be factorized as $p=\psi q$, where
  $q:A\to B$ is the cokernel of $j: [A,K]=[A,\K] \to A$ and $\psi: B\to C$ is given by the universal property of cokernels. Since $\K$ is normal in $A$,
   according to Theorem 2.8.11 of \cite{bibbia} we can say that $\psi: B=A/[A,\K]\to C=A/\K$ is central, that is
  $[B,N]=0, $ where $n:N\to B$ denotes the kernel of $\psi.$ In other words, this means that there exists a cooperator $\theta: B\times N \to B,$ such that
  $$ \theta \langle 0,1\rangle=n \qquad {\text{and}¥} \qquad  \theta \langle 1,0\rangle=1_{B}¥$$
 or, equivalentely, that the diagram
 $$
 \xymatrix{
B\times N\ar[r]^{\theta}
\ar[d]_{\pi_{B}¥}
&B\ar[d]^{\psi}\\
B\ar[r]_{\psi}
&C
}.
 $$
 is a pullback diagram.
 Then, in the following diagram
 $$
 \xymatrix{
A\times N\ar[r]^{q\times 1}\ar[d]_{\pi_{A}}&B\times N\ar[r]^{\theta}
\ar[d]_{\pi_{B}}
&B\ar[d]^{\psi}\\
A\ar[r]_{q}&B\ar[r]_{\psi}
&C
}.
 $$

 the outer rectangle is a pullback of $p=\psi q$ along $\psi.$ But,

$$\newdir{|>}{%
!/4.5pt/@{|}*:(1,-.2)@^{>}*:(1,+.2)@_{>}} \xymatrix@C=10ex{
A+K\ar@{}[dr]|{}\ar@{->>}[r]^{[1,k]}\ar@{->>}[d]^{\Sigma}\ar@/_5ex/@{->>}[dd]_{[1,0]}&A\ar@{->>}[d]^q \ar@/^5ex/@{->>}[dd]^p
\\
A\times
K\ar@{->>}[r]^{\varphi}\ar@{->>}[d]^{\pi_A}&B\ar@{->>}[d]^{\psi}
\\
A\ar@{->>}[r]_p&C.
}
$$
 since the outer and the top rectangles in the above diagram are pushouts,  the bottom one is a pushout of regular epimorphisms and then a regular pushout,
 as it happens in any semi-abelian category. This means that the morphism  $\alpha: A\times K\to A\times N $ given by the universal property of the pullback:

 $$
 \xymatrix{
 A\times K\ar[dr]_{\alpha}\ar[dd]_{\pi_{A}}\ar[rrr]^{\varphi}&&&B\ar[dd]^{\psi}\\
&A\times N\ar[r]^{q\times 1}\ar[dl]_{\pi_{A}}&B\times N\ar[ru]^{\theta}
\ar[dd]^{\pi_{B}}
&\\
A\ar[rrd]^{q}\ar[rrr]^{p}&&&C\\
&&B\ar[ru]_{\psi}
}.
 $$

 is a regular epimorphism. Furthermore,   ${\theta}({q\times 1})\alpha=\varphi$ and  ${\alpha \pi_{A}=\pi_{A}}. $ In order to apply  Lemma \ref{alfa}, we need only
 to show that $\alpha \langle 1,0\rangle=\langle 1,0\rangle$. Since are both morphisms to a pullback, it is sufficient to show that the equality holds when composing
  with the projections.
$$(\pi_{A}\alpha)\langle 1,0\rangle=\pi_{A}\langle 1,0\rangle=1_{A}=\pi_{A}\langle 1,0\rangle_{N}$$
$$\theta(q\times 1)\alpha\langle 1,0\rangle=\varphi\langle 1,0\rangle=\varphi\Sigma i_{A}=q[1,k]i_{A}=q=
\theta\langle 1,0\rangle q=\theta(q\times 1)\langle 1,0\rangle $$

the last equalities holding, by the properties of the cooperator $\theta$.
 Applying  Lemma \ref{alfa}, we get $\alpha=1\times r$, with $r$ regular epimorphism, since $\alpha$ is.
 Hence $ (q\times 1)\alpha=(q\times 1)(1\times r)=(q\times r)$  and then the following diagram commutes:

 $$
 \xymatrix{
K\ar[r]^{\langle 0,1\rangle}\ar[d]_{r}&A\times K\ar[d]^{q\times r}\ar[dr]^{\varphi}&\\
N\ar[d]_{¥}\ar[r]^{\langle 0,1\rangle}&B\times N\ar[r]^{\theta}
\ar[d]_{\pi_{B}}&B\ar[d]^{\psi}\\
0\ar[r]_{¥}&B\ar[r]_{\psi}
&C
}.
 $$
 Since the two bottom squares are pullbacks, $\theta\langle 0,1\rangle=n$. Furthermore, $\varphi\langle 0,1\rangle=\varphi\Sigma i_{K}=q[1,k]i_{K}=qk$;
 in conclusion we have the following commutative diagram:
 $$
 \xymatrix{
K\ar[d]^{k}\ar@{->>}[r]^r&N\ar[d]^{n}\\
A\ar@{->>}[r]^{q}&B
}.
 $$

 where the horizontal arrows are regular epimorphisms, the rightmost one is a normal monomorphism and the leftmost one is a monomorphism.
 But Ker$(q )=[A,K] \leq K$ by hypothesis, so we can apply Hoffmann's axiom and concude that $K$ is normal in $A.$¥

 \end{proof}

\bibliographystyle{alpha}

\bibliography{normality}

\end{document}